\documentclass[review]{article}

\usepackage{amsmath, amsthm, amssymb}
\usepackage[utf8]{inputenc}
\usepackage{amsfonts}
\usepackage{amssymb}
\usepackage{enumitem}
\usepackage{amsthm}
\usepackage{dsfont}
\usepackage{natbib}
\bibpunct[, ]{[}{]}{,}{n}{,}{,}
\makeatother
\theoremstyle{plain}
\DeclareMathAlphabet{\mathbbmsl}{U}{bbm}{m}{sl}

\def\min{\mathop{\rm min}}

\newcommand\rr{\mathbbmsl{R}}

\def\sol{\mathop{\rm SOL}}
\def\min{\mathop{\rm min}}

\def\a{\alpha}
\def\A{{\mathcal{A}}}
\def\L{{\rm LCP}}
\def\T{{\rm TCP}}

\def\R{\mathbb{R}^n}

\def\w{\bf w}

\def\B{{\mathcal{B}}}

\def\w{{\bf w}}

\def\x{{\bf{x}}}
\def\y{{\bf{y}}}

\def\q{\bf{q}}
\def\u{{\bf{u}}}
\def\z{\bf{z}}

\def\A{\mathcal{A}}

\newtheorem{theorem}{Theorem}

\newtheorem{corollary}{Corollary}
\newtheorem{example}{Example}
\newtheorem{Que}{Question}

\newtheorem{rem}{Remark}
\newtheorem{lemma}{Lemma}
\newtheorem{definition}{Definition}
\newtheorem{proposition}{Proposition}
\begin{document}\begin{center}
{\large A Criterion for ${\rm Q}$-tensors}
\end{center}
\begin{center}
{\rm SONALI SHARMA}\\
Department of Mathematics,\\
Malaviya National Institute of Technology Jaipur, Jaipur, 302017, India\\
E-mail address: Ssonali836@gmail.com\\ 
{\rm K. PALPANDI}\\
Department of Mathematics,\\
Malaviya National Institute of Technology Jaipur, Jaipur, 302017, India\\
E-mail address: kpalpandi.maths@mnit.ac.in \\
\end{center}

\begin{abstract}
A tensor ${\A}$ of order $m$ and dimension $n$ is called a ${\rm Q}$-tensor if the tensor complementarity problem has a solution for all ${\q} \in {\R}$. This means that for every vector ${\q}$, there exists a vector ${\u}$ such that ${\u} \geq {\bf 0},{\w} = {\A}{\u}^{m-1}+{\q} \geq {\bf 0},~\text{and}~ {\u}^{T}{\w} = 0$. In this paper, we prove that within the class of rank one symmetric tensors, the ${\rm Q}$-tensors are precisely the positive tensors. Additionally, for a symmetric ${\mathrm Q}$-tensor ${\A}$ with $rank({\A})=2$, we show that ${\A}$ is an ${\mathrm R}_{0}$-tensor. The idea is inspired by the recent work of Parthasarathy et al. \cite{Parthasarathy} and Sivakumar et al. \cite{Sivakumar} on ${\rm Q}$-matrices.
\end{abstract}

\section{Introduction}
For a given real square matrix ${\bf A}$ of order $n$ and a vector ${\q}$ in ${\R}$, the linear complementarity problem \cite{Cottle}, denoted as ${\rm LCP} ({\bf A},{\q})$, seeks to find a vector ${\u}$ in ${\R}$  that satisfies 
$${\u} \geq {\bf 0}, {\w} = {\bf A}{\u}+{\q} \geq {\bf 0}~\text{and}~{\u}^{T}{\w}=0,$$
or to show that no such vector exists. If such a vector exists for every ${\q}$ in ${\R}$, then ${\bf A}$ is called a ${\rm Q}$-matrix.  
On the other hand, if the only solution to ${\rm LCP}({\bf A},{\bf 0})$ is zero, then ${\bf A}$ is known as an ${\mathrm R}_{0}$-matrix. The class of ${\rm Q}$-matrices is significant in the study of the linear complementarity problem, as it characterizes the solvability of the ${\rm LCP}({\bf A},{\q})$. The literature has extensively investigated this class, as seen in \cite{Cottle,murty}.


Tensors, as a generalization of matrices, are gaining more and more significance in the current big data era. Tensors  are also referred as hyper matrices. Given natural numbers $m$ and $n$, a real tensor ${\A}$ of order $m$ and dimension $n$ is expressed as ${\A} = (a_{i_{1}i_{2}...i_{m}})$, where $i_{j} \in [n]$, for all $j \in [m]$. The tensor complementarity problem ($\T$), as an extension of the linear complementarity problem, was defined by Song and Qi \cite{Song}. The tensor complementarity problem, $\T({\A},{\q})$, associated with a tensor ${\A}$ of order $m$ and dimension $n$  and a vector ${\q} \in {\R}$, is to find a vector ${\u}$ such that ${\u} \geq {\bf 0},{\w} = {\A}{\u}^{m-1}+{\q} \geq {\bf 0},~\text{and}~ {\u}^{T}{\w} = 0$. Here $$({\A}{\bf u}^{m-1})_i = {\sum_{i_{2},...,i_{m}=1}^n}a_{i i_{2}...i_{m}}u_{i_{2}}...u_{i_{m}},$$ is a homogeneous polynomial of degree $m-1$. The concept of ${\rm Q}$-matrices has been extended to ${\rm Q}$-tensors by Song and Qi \cite{Yisheng}, where a tensor ${\A}$ is a ${\mathrm Q}$-tensor if $\T({\A},{\q})$ has a solution for each ${\q} \in {\R}$. ${\A}$ is said to be an ${\mathrm R}_{0}$-tensor if $\T({\A},{\bf 0})$ has zero as the only solution. Recently, many results related to the linear complementarity problem and corresponding matrix classes have been extended to the tensor complementarity problem. For more information, we refer to \cite{{Balaji},{Ge},{Mei},{Meng},{Palpandi}}. The survey articles \cite{{Qi.L 1},{Qi.L 2},{Qi. L 3}} provides a good summary of the current status of the theoretical, applied, and algorithmic research on TCP.
 Among other problems considered in \T, characterizing the class of ${\mathrm Q}$-tensors is one of the key problem. Many researchers have extended the  ${\mathrm Q}$-matrices characterization to ${\mathrm Q}$-tensors. Relevant literature  can be found in \cite{{Huang},{Yisheng}}.
 
 Recently, Sivakumar et al. \cite{Sivakumar} proved that if a rank one matrix is a ${\mathrm Q}$-matrix, then it is positive (A matrix with all  its entries being positive). Another result by Parthasarathy  et al. \cite{Parthasarathy} states that any ${\mathrm Q}$-matrix with order less than or equal to $2$ is an ${\mathrm R}_{0}$-matrix. These results have led to the question of whether similar characterizations exist for ${\mathrm Q}$-tensors in the case of \T. The main objective of this paper is to answer these questions.
 
The paper addresses two problems related to ${\mathrm Q}$-tensors. Firstly, it presents counter examples to show that for a rank one tensor, a ${\mathrm Q}$-tensor, may not necessarily be a positive tensor (A tensor having all  its entries positive). Secondly, it establishes the equivalence of ${\mathrm Q}$, ${\mathrm R}_{0}$, and positive for the class of rank one symmetric tensors. At last, we prove that if a symmetric tensor ${\A}$ of order $m$ and dimension $2$ has rank $2$ and is a ${\mathrm Q}$-tensor, then it is also an ${\mathrm R}_{0}$-tensor.\\
Preliminary material, notation and definitions that are useful for our manuscript are listed out in Section 2. The main results related to  rank one symmetric tensors are given in Subsection 3.1 and rank two symmetric tensors are provided in Subsection 3.2. Concluding remarks are  given in Sect. 4.
\section{Preliminaries}
In this section, we give some definitions and results that will be helpful to us throughout the paper.\\
 The $n$-dimensional Euclidean space equipped with the usual inner product is denoted by ${\R}$. The scalars are written as small letters such as $\alpha, t, a$,.. and the vectors in ${\R}$ are written as bold letters such as ${\x},{\y}$.  For any ${\x} \in {\R}$, $x_{i}$ denotes the $i$-th component of ${\x}$ and ${\x}^{T}$ is the transpose of ${\x}$. ${\bf 0}$ is used to denote the zero vector in ${\R}$. We write ${\x} \geq {\bf 0}~ (> {\bf 0})$, if each component of ${\x}$ is non-negative (positive). $[n]$ is used to denote the index set $\{1,2,...,n\}$. ${\bf e ^{(i)}}$ is a vector in ${\R}$ having $1$ at the $i$-th place and rest zero. For any ${\x} ,{\y}$ in ${\R}$, ${\x} \otimes {\y}$ is the outer product of ${\x}$ and ${\y}$, that is, ${\x} \otimes {\y} = (x_{i} y_{j})$, for each $i,j \in [n]$. We say that a vector ${\w} = (w_{1},w_{2},...,w_{n}) \in {\R}$ such that $w_{i}$ is non-zero, for each $i \in [n]$, is unisigned if either ${\w} > {\bf 0}$ or ${\w} < {\bf 0}$. ${\mathbb R}^{n \times n}$ is the set of all real square matrices of order $n$.  Any matrix in ${\mathbb R}^{n \times n}$ is denoted by bold capital letters e.g. ${\bf A}$, ${\bf P}$. ${\bf I}$ denotes the identity matrix in ${\mathbb R}^{n \times n}$. $\mathrm{T}(m,n)$ is the set of all $m$th order $n$ dimensional real tensor. Any tensor in  $\mathrm{T}(m,n)$ is written as calligraphic letters e.g. ${\A},{\B}$. The entries of a tensor ${\A}$ are denoted as $a_{i_{1}i_{2}...i_{m}}$, where $i_{j} \in [n]$ and $j \in [m]$. For any ${\A} \in \mathrm{T}(m,n)$ and ${\u} \in {\R}$,
$$({\A}{\bf u}^{m-1})_i = {\sum_{i_{2},...,i_{m}=1}^n}a_{i i_{2}...i_{m}}u_{i_{2}}...u_{i_{m}}.$$ 
For any ${\A} \in {\rm T}(m,n)$ and ${\q} \in {\R}$, the tensor complementarity problem, denoted as $\T({\A},{\q})$, is to find a vector ${\u} \in {\R}$ such that
$${\u} \geq {\bf 0}, {\A}{\u}^{m-1}+{\q} \geq {\bf 0}~\text{and}~{\u}^{T}({\A}{\u}^{m-1}+{\q})=0,$$
or to show that no such vector exists. The vector ${\u}$ is said to be a solution of the $\T({\A},{\q})$.  The set of all such vectors is denoted by $\sol({\A},{\q})$ and refereed as the solution set of the  $\T({\A},{\q})$.
\begin{definition}\rm
A tensor ${\A} \in {\mathrm T} (m,n)$ is said to be
\begin{enumerate}
\item[\rm (i)]  a positive tensor, if all its entries are positive.
\item[\rm (ii)] an ${\rm S}$-tensor \cite{Gaohang}, if there exists a vector ${\w} > {\bf 0}$ such that ${\A}{\w}^{m-1} > {\bf 0}$.
\item[\rm (iii)] a rank one tensor \cite{Qi. L 4}, if there exist non-zero vectors ${\x}^{(1)}, {\x}^{(2)},...,{\x}^{(m)}$ in ${\R}$ such that
$${\A} = {\x}^{(1)}\otimes{\x}^{(2)}\otimes...\otimes{\x}^{(m)}.$$
\end{enumerate}
It is apparent that any tensor can always be decomposed as a linear combination of rank one tensors, although the decomposition need  not be  unique, see \cite[Section 4]{comon}. The minimum number of rank one tensors required to decompose a  tensor  ${\A} \in {\rm T}(m,n)$  as a linear combination of rank one tensors is known as its rank and it is denoted by $rank({\A})$.
\end{definition}

\subsection{Symmetric tensors}
A tensor ${\A}$ is said to be a symmetric tensor if it is invariant under any permutation of its indices. ${\rm S}(m,n)$ is used to denote the set of all $m$th order $n$ dimensional real symmetric tensors.
\begin{definition}\rm A tensor ${\A} \in T(m,n)$ is said to be a {\it rank one symmetric tensor} if there exists a non-zero vector ${\w}\in {\R}$ such that 
$$a_{i_{1}i_{2}...i_{m}} = w_{i_{1}}w_{i_{2}}...w_{i_{m}},~\text{where}~i_{1},...,i_{m}\in [n].$$
\end{definition}
The vector ${\w}$ is said to be the generator of $\A$, and we write ${\A} = [{\w}]^{{\bigotimes m}}$.\\ 

The following theorem  can be considered as a generalization of the eigenvalue decomposition for symmetric matrices to symmetric tensors.
\begin{theorem}\rm \citep[][Lemma 4.2]{comon}\label{decomposition}
Let ${\A} \in \mathrm{S}(m,n)$. Then there exists scalars ${\mu_{1}}, {\mu_{2}},...,{\mu_{k}}$ in ${\mathbb{R}}$ and vectors ${\w}^{(1)},{\w}^{(2)},..., {\w}^{(k)}$ in ${\R}$ such that 
$\A= \sum_{j=1}^k \mu_j [{\w^{(j)}}]^{\bigotimes m}.$
\end{theorem}

The representation of any symmetric tensor as a linear combination of rank one symmetric tensors, is also known as symmetric CP decomposition \cite{Kolda}, symmetric outer product decomposition \cite{comon} etc. It should be noted that the decomposition of a symmetric tensor as given in Theorem \ref{decomposition}, need not be unique. However one can always write a symmetric tensor in the form given in Theorem \ref{decomposition} and therefore we define the symmetric rank of a symmetric tensor as follows:
\begin{definition}\rm 
 Let $\A\in T(m,n)$ be a symmetric tensor. The {\it symmetric rank} of $\A$ is denoted by  $\mathop{Sym(\A)}$, defined by
 $$\mathop{Sym(\A)}:=\min\{r: \A=\sum_{j=1}^r \mu_j [{ \w^{(j)}}]^{\bigotimes m},~\mu_j \in \mathbb{R},~{\w^{(j)}} \in \R\}.$$ 
\end{definition}  

The following theorem, which gives a characterization of ${\rm Q}$-tensors will be required to provide counterexamples.
\begin{theorem}\rm\cite[Theorem 2.2]{Huang} \label{Q tensor}
Let ${\A} = (a_{i_{1}i_{2}...i_{m}}) \in T(m,n)$ where $i_{1},i_{2},...,i_{m}\in [n]$. Denote ${\A}_{1...1}:= (a_{i_{1}i_{2}...i_{m}}) \in T(m,n-1)$ with $i_{1},i_{2}...,i_{m}\in [n]\setminus \{1\}$ and ${\A}_{2...2}:= (a_{i_{1}i_{2}...i_{m}}) \in T(m,n-1)$ with $i_{1},i_{2},...,i_{m}\in [n]\setminus \{2\}$. Suppose that $a_{1i_{2}...i_{m}} = a_{2i_{2}...i_{m}}$ for all $i_{2},...,i_{m}$ in $[n]$, and both ${\A}_{1...1}$ and ${\A}_{2...2}$ are ${\rm Q}$-tensors. Then $\A$ is a ${\rm Q}$-tensor.
\end{theorem}

\section{Main Results}
\subsection{Rank one symmetric tensors}
In this subsection, we show  that the well-known result on ${\mathrm{Q}}$-matrices \citep[][Theorem 3.3]{Sivakumar} that a rank one matrix ${\bf A}$ is a ${\mathrm{Q}}$-matrix if and only if ${\bf A}$ is positive, may not be true for tensors.  Clearly, any positive tensor must be a ${\mathrm{Q}}$-tensor, \citep[][Theorem 3.5]{Yisheng}. However, the converse does not hold. In the following, we give an example of even and  odd order tensors which are ${\rm Q}$-tensors but not positive.

\begin{example}\rm  Let ${\A} \in T(3,2)$ such that ${\A} = {\x}\otimes{\y}\otimes{\bf z}$, where ${\x}=(-1,-1), {\y} = (1,-1), {\z}=(-1,1)$. Then we have $a_{111}=1, a_{112}=-1, a_{121}=-1, a_{122}=1, a_{222}=1, a_{211}=1, a_{212}=-1, a_{221}=-1$. So ${\A}$ is not a positive tensor. Now we show that ${\A}$ is a ${\mathrm{Q}}$-tensor. We have ${\A}_{111}=(a_{222}) =1>0$ and ${\A}_{222}=(a_{111}) =1>0$. Therefore both the principal subtensors ${\A}_{111}$ and ${\A}_{222}$ are ${\mathrm{Q}}$-tensors. Also note that $a_{1i_{2}i_{3}} = a_{2i_{2}i_{3}}$, for all $i_{2}, i_{3} \in \{1,2\}$. Hence by Theorem \ref{Q tensor}, $\A$ is a ${\mathrm{Q}}$-tensor. 
\end{example}
\begin{example}\rm Let ${\A} \in T(4,2)$ such that ${\A} = {\x}\otimes{\y}\otimes{\bf z}\otimes{\w}$, where ${\x}=(1,1), {\y} = (1,-1), {\z}=(-1,-1), {\w} = (-1,1)$. Then we have $a_{1111}=1, a_{2222}=1, a_{1112}=-1, a_{2112}= -1, a_{1121}=1, a_{2121}=1, a_{1211}=-1, a_{2211}=-1, a_{1122}=-1, a_{2122}=-1, a_{1222}=1, a_{2111}=1, a_{1212}=1, a_{2212}=1, a_{1221}=-1, a_{2221}=-1$. So ${\A}$ is not a positive tensor. Now we show that ${\A}$ is a ${\mathrm{Q}}$-tensor. Note that ${\A}_{1111}=(a_{2222}) =1>0$ and ${\A}_{2222}=(a_{1111}) =1>0$. Also note that $a_{1i_{2}i_{3}i_{4}} = a_{2i_{2}i_{3}i_{4}}$, for all $i_{2}, i_{3}, i_{4} \in \{1,2\}$. Hence by Theorem \ref{Q tensor}, $\A$ is a ${\mathrm{Q}}$-tensor.   
\end{example}

We now give a property of rank one symmetric tensors.
\begin{proposition}\label{positive} \rm Let ${\A} \in {\mathrm T}(m,n)$ be a rank one symmetric tensor. The followings are true.
\begin{enumerate}
\item[\rm (i)]  If $\A$ is positive, then the generator of $\A$  is unisigned.
\item[\rm (ii)] If the generator of ${\A}$ is unisigned, then either ${\A}$ is positive or $-{\A}$ is positive.
\end{enumerate}
\end{proposition}
\begin{proof}
Let ${\A} \in {\mathrm T}(m,n)$ be a rank one symmetric tensor such that ${\A} = [{\w}]^{\bigotimes m}$.\\ 
(i): Let $\A$ be a positive tensor. It is clear that $w_{i}$ is non-zero for each $i \in [n]$. Note that $a_{ij...j} = w_{i}{w_{j}}^{m-1}$. Since ${\A}$ is positive, therefore $w_{i}{w_{j}}^{m-1} > 0$, for each $i,j \in [n]$. If $m$ is odd, then  $w_{i} > 0$, for each $i \in [n]$. Hence ${\w} > {\bf 0}$. Also, if  $m$ is even, then $w_{i}$ and $w_{j}$ must have the same sign, for each $i,j\in [n]$. Hence either ${\w} > {\bf 0}$ or ${\w}< {\bf 0}$. Therefore the generator of ${\A}$  is unisigned.\\
(ii): Suppose that the generator of ${\A}$ is unisigned, so $w_{i}$ is non-zero, for each $i \in [n]$. If ${\w}> {\bf 0}$, then we get ${\A}$ is a positive tensor. Also when ${\w}<{\bf 0}$ and $m$ is even gives us ${\A}$ is positive. If ${\w} < {\bf 0}$ and $m$ is odd, then all the entries of ${\A}$ are negative. Hence $-{\A}$ is positive.
\end{proof}

We now prove our result for a rank one symmetric tensor.
\begin{theorem} \rm \label{Equivalent} Let ${\A} \in {\mathrm{T}}(m,n)$ be a rank one symmetric tensor. The followings are equivalent.
\begin{enumerate}
\item[(i)] $\A$ is a ${\rm Q}$-tensor.
\item[(ii)] ${\A}$ is an ${\rm S}$-tensor.
\item[(iii)] $\A$ is positive.
\item[(iv)] $\A$ is an ${\rm R}_{0}$-tensor.
\end{enumerate}
\end{theorem}
\begin{proof}
Let ${\A} \in {\mathrm{T}}(m,n)$ be a rank one symmetric tensor such that ${\A} = [{\w}]^{\bigotimes m}$. The direction of the proof is as follows: (i)$\implies$(ii)$\implies$(iii)$\implies$(i) and (iii)$\iff$ (iv). \\ 
(i)$\implies$(ii): It follows from \cite[Theorem 3.1]{Gaohang}.\\
(ii)$\implies$(iii): Let ${\A}$ be an ${\rm S}$-tensor. We first claim that $w_{i}$ is non-zero, for each $i \in [n]$. Assume to the contrary, that there exists an  index $j$ such that $w_{j} = 0$. Then for each ${\u} > {\bf 0}$, we have $({\A}{\u}^{m-1})_{j} = w_{j}({\w}^{T}{\u})^{m-1}=0$, a contradiction to ${\A}$ being an ${\rm S}$-tensor. Hence $w_{i}$ is non-zero, for each $i \in [n]$. We now prove that ${\A}$ is positive. Since ${\A}$ is an ${\rm S}$-tensor, the system
$${\A}{\u}^{m-1}>{\bf 0}, {\u}>{\bf 0},$$
has a solution. So there exists ${\u}>{\bf 0}$ such that $({\A}{\u}^{m-1})_{i} = w_{i}({\w}^{T}{\u})^{m-1} > {0}$, for each $i \in [n]$. If $m$ is odd, then $({\w}^{T}{\u})^{m-1} > 0$ implying $w_{i}>0$, for each $i \in [n]$. Hence ${\w} > {\bf 0}$. If $m$ is even, then we get $w_{i}$ and  $({\w}^{T}{\u})^{m-1}$ must have the same sign, for each $i \in [n]$. Hence ${\w}$ is unisigned. By Proposition \ref{positive}, either $\A$ is positive or $-{\A}$ is positive. Since $\A$ is an $\rm {S}$-tensor, ${\A}$ must be positive.\\
(iii)$\implies$(i): The proof follows from \citep[][Theorem 3.5]{Yisheng}.\\
(iii)$\implies$(iv):  Clearly, any positive tensor is an ${\rm R}_{0}$-tensor.\\
(iv)$\implies$(iii): Let $\A$ be an ${\rm R}_{0}$-tensor. We first claim that $w_{i}$ is non-zero, for each $i \in [n]$. Assume contrary. Suppose there exists an index $i$ such that $w_{i} = 0$. Let ${\y} = {\bf e ^{(i)}}$. It is easy to see that ${\y}$ is a non-zero solution of ${\T}({\A},{\bf 0})$, which is a contradiction to ${\A}$ being an ${\rm R}_{0}$-tensor. Hence $w_{i}$ is non-zero, for each $i \in [n]$. We now show that ${\A}$ is positive. To the contrary, assume that there exist  $i_{1},i_{2},...,i_{m} \in [n]$ such that $a_{i_{1}i_{2}...i_{m}}<0$. This implies $w_{i_{1}}w_{i_{2}}...w_{i_{m}} < 0$. Then there must exist atleast two indices $i,j \in [n]$ such that $w_{i}w_{j} < 0$. Without loss of generality, assume that $w_{1}<0$ and $w_{2}>0$. Then it can be easily verified that ${\z} = (w_{2},-w_{1},...,0)$ is a non-zero solution of ${\T}({\A},{\bf 0})$, which is a contradiction to $\A$ being an ${\rm R}_{0}$-tensor. Hence $\A$ is positive.  
\end{proof}

\subsection{Symmetric tensors with ${\mathop{Sym}({\A}) = 2}$}
Recently, Parthasarathy et al. \cite[][Theorem 3.3]{Parthasarathy} proved that if ${\bf A}$ is an $n \times n$ real ${\mathrm Q}$-matrix, then it is ${\mathrm{R}}_{0}$ for $n \leq 2$. This raises a question whether the same result holds in the case of tensors. For $n=1$, the tensor ${\A}$ contains a single positive entry, being a ${\rm Q}$-tensor and hence it is ${\rm R}_{0}$. However, it may not hold for $n=2$. Here is an example. 
\begin{example}\rm Let ${\A} \in T(3,2)$, where $a_{111}= 4 = a_{211}, a_{222}=1=a_{122}, a_{121}=-4=a_{221}$ and $a_{112} = 0 = a_{212}$. Then it can be easily verified using Theorem \ref{Q tensor} that $\A$ is a ${\mathrm Q}$-tensor, but $\A$ is not an ${\mathrm{R}}_{0}$-tensor. This can be seen as follows: 
$${\A}{\x}^{2} = \begin{pmatrix}
4x_{1}^{2}+x_{2}^{2}-4x_{1}x_{2}\\
4x_{1}^{2}+x_{2}^{2}-4x_{1}x_{2}
\end{pmatrix}$$
Take ${\x} = (1,2)$, then ${\A}{\x}^{2} = (0,0)$. So ${\T}({\A},{\bf 0})$ has a non-zero solution.
\end{example}
In the following, we prove that the result holds if we put an extra condition on the tensor ${\A}$ to be a symmetric tensor with $\mathop{Sym(\A)} =2$. We start by proving the following results, which will provide crucial information to prove our main result.

\begin{theorem}\rm \label{LCP}
Let ${\A} \in {\rm S}(m,n)$  and ${\q} \in {\R}$. Let ${\A} =  \sum_{j=1}^k \mu_j [{\w^{(j)}}]^{\bigotimes m}$ be the symmetric outer product decomposition of ${\A}$, for some scalars ${\mu_{1}}, {\mu_{2}},...,{\mu_{k}}$ and vectors ${\w}^{(1)},{\w}^{(2)},..., {\w}^{(k)}$ in ${\R}$. Then ${\u} \in \sol({\A},{\q})$ if and only if ${\bf P}{\u} \in \sol(\hat{\A},{\bf P}{\q})$, where ${\bf P}$ is a permutation matrix in ${\mathbb R}^{n \times n}$ and  ${\hat{\A}} = \sum_{j=1}^k \mu_j [{\bf P}{\w^{(j)}}]^{\bigotimes m} \in {\rm S}(m,n)$.
\end{theorem}
\begin{proof}
Given that ${\A} =  \sum_{j=1}^k \mu_j [{\w^{(j)}}]^{\bigotimes m} \in {\rm S}(m,n)$ is the symmetric outer product decomposition of ${\A}$, where  ${\mu_{1}}, {\mu_{2}},...,{\mu_{k}} \in {\mathbb{R}}$ and ${\w}^{(1)},{\w}^{(2)},..., {\w}^{(k)} \in {\R}$. Let ${\q} \in {\R}$ and ${\bf P}$ be a permutation matrix in ${\mathbb R}^{n \times n}$. Let ${\u} \in \sol({\A},{\q})$. Then we have ${\bf u} \geq {\bf 0}$, ${\A}{\u}^{m-1}+{\q} \geq {\bf 0}$ and ${\u}^{T}({\A}{\u}^{m-1}+{\q}) = 0$. Note that  
$$({\A}{\u}^{m-1})_{i} = \mu_{1} {w^{(1)}_{i}}(({\w}^{(1)})^{T}{\u})^{m-1}+ \mu_{2} {w^{(2)}_{i}}(({\w}^{(2)})^{T}{\u})^{m-1}+...+\mu_{k} {w^{(k)}_{i}}(({\w}^{(k)})^{T}{\u})^{m-1}.$$
Therefore 
$$({\A}{\u}^{m-1}+{\q}) = \begin{bmatrix}
\mu_{1}{\w}^{(1)} & \mu_{2}{\w}^{(2)}&...&\mu_{k}{\w}^{(k)}
\end{bmatrix} \begin{bmatrix}
(({\w}^{(1)})^{T}{\u})^{m-1}\\
(({\w}^{(2)})^{T}{\u})^{m-1}\\
\vdots\\
(({\w}^{(1)})^{T}{\u})^{m-1}
\end{bmatrix} + {\q}.$$
This implies
\begin{equation} \label{PA}
{\bf P}({\A}{\u}^{m-1}+{\q})= \begin{bmatrix}
\mu_{1}{\bf P}{\w}^{(1)} & \mu_{2}{\bf P}{\w}^{(2)}&...&\mu_{k}{\bf P}{\w}^{(k)}
\end{bmatrix} \begin{bmatrix}
(({\bf P}{\w}^{(1)})^{T}({\bf P}{\u}))^{m-1}\\
(({\bf P}{\w}^{(2)})^{T}({\bf P}{\u}))^{m-1}\\
\vdots\\
(({\bf P}{\w}^{(k)})^{T}({\bf P}{\u}))^{m-1}
\end{bmatrix} + {\bf P}{\q}.
\end{equation}
Let $\hat{\A} = \sum_{j=1}^k \mu_j [{\bf P}{\w^{(j)}}]^{\bigotimes m}$, then from Eq.(\ref{PA}), we can see that ${\bf P}({\A}{\u}^{m-1}+{\q}) = {\hat{\A}}({\bf P}{\u})^{m-1} + {\bf P}{\q}$. Therefore we get ${\bf P}{\u} \geq {\bf 0}, {\hat{\A}}({\bf P}{\u})^{m-1} + {\bf P}{\q} \geq {\bf 0}$, and $({\bf P}{\u})^{T}({\hat{\A}}({\bf P}{\u})^{m-1} + {\bf P}{\q})= {\u}^{T}({\A}{\u}^{m-1}+{\q}) = 0$. Hence ${\bf P}{\u}$ is a solution of $\T({\hat{\A}},{\bf P}{\q})$. The converse follows similarly.
\end{proof}

As a consequence, we have the following result.
\begin{corollary}\rm \label{not}
Let ${\A} \in {\rm S}(m,n)$ and ${\A} =  \sum_{j=1}^k \mu_j [{\w^{(j)}}]^{\bigotimes m}$ be the symmetric outer product decomposition of ${\A}$, where  ${\mu_{1}}, {\mu_{2}},...,{\mu_{k}} \in {\mathbb{R}}$ and ${\w}^{(1)},{\w}^{(2)},..., {\w}^{(k)} \in {\R}$. Let ${\hat{\A}} = \sum_{j=1}^k \mu_j [{\bf P}{\w^{(j)}}]^{\bigotimes m} \in {\rm S}(m,n)$, and ${\bf P}$ be a permutation matrix in ${\mathbb R}^{n \times n}$. Then
\begin{enumerate}
\item[\rm (i)] ${\A}$ is an ${\rm R}_{0}$-tensor if and only if $\hat{\A}$ is an ${\rm R}_{0}$-tensor.
\item[\rm (ii)] ${\A}$ is a ${\rm Q}$-tensor if and only if $\hat{\A}$ is a ${\rm Q}$-tensor. 
\end{enumerate}
\end{corollary}

\begin{lemma} \rm \label{symrank2}
Let ${\A} \in {\mathrm{S}}(m,2)$ with $\mathop{Sym(\A)} =2$. Let ${\A} = {\mu_{1}}{\x}^{\bigotimes {m}} + {\mu_{2}}{\y}^{\bigotimes {m}}$ be the symmetric outer product decomposition of ${\A}$, where $\mu_{1}, \mu_{2} \in \mathbb{R}$ and ${\x},{\y} \in {\mathbb{R}}^{2}$. Then ${\x}$ and ${\y}$ are  linearly independent.
\end{lemma}
\begin{proof}
Let ${\A} \in {\mathrm{S}}(m,2)$  be such that $\mathop{Sym(\A)} =2$. Let ${\A} = {\mu_{1}}{\x}^{\bigotimes {m}} + {\mu_{2}}{\y}^{\bigotimes {m}}$ be the symmetric outer product decomposition of ${\A}$, where  $\mu_{1}, \mu_{2} \in \mathbb{R}$ and ${\x},{\y} \in {\mathbb{R}}^{2}$. Since $\mathop{Sym(\A)} =2$, therefore ${\x}$ and ${\y}$ must be non-zero as well as $\mu_{1}, \mu_{2}$ are non-zero. Assume contrary that ${\y} ={\alpha}{\x}$, for some non-zero ${\alpha}$. This implies ${\A} = {\beta}{\x}^{\bigotimes {m}}$, where $\beta= (\mu_{1}+\mu_{2}{\alpha}^{m})$ and hence $\mathop{Sym(\A)} = 1$, a contradiction. Therefore ${\x}$ and ${\y}$ must be linearly independent.  
\end{proof}

\begin{lemma}\rm\label{nonzero}
Let ${\A} \in {\mathrm{S}}(m,2)$ with $\mathop{Sym(\A)} =2$. Then any non-zero solution of $\T({\A},{\bf 0})$ has exactly one non-zero component.
\end{lemma}
\begin{proof}
Let ${\A} \in {\mathrm{S}}(m,2)$  with $\mathop{Sym({\A})}=2$. From Theorem \ref{decomposition}, it follows that there exists non-zero scalars $\mu_{1}, \mu_{2}$ in $\mathbb{R}$ and non-zero vectors ${\x},{\y}$ in ${\mathbb{R}}^{2}$ such that 
\begin{equation}
{\A} = {\mu_{1}}{\x}^{\bigotimes {m}} + {\mu_{2}}{\y}^{\bigotimes {m}}.
\end{equation}
From Lemma \ref{symrank2}, we get ${\x}$ and ${\y}$ are linearly independent. Let ${\u}=(u_{1},u_{2})^{T}$ be a non-zero solution to $\T({\A},{\bf 0})$. Assume contrary that $u_{1}> 0,~\text{and}~u_{2}>0$. This implies $({\A}{\u}^{m-1})_{1} = 0,~\text{and}~({\A}{\u}^{m-1})_{2} = 0$. From this, we get $\mu_{1}x_{1}({\x}^{T}{\u})^{m-1} + \mu_{2}y_{1}({\y}^{T}{\u})^{m-1}=0$ and $\mu_{1}x_{2}({\x}^{T}{\u})^{m-1} + \mu_{2}y_{2}({\y}^{T}{\u})^{m-1}=0$. This can be equivalently written as 
 $$\begin{bmatrix}
 \mu_{1}x_{1} & \mu_{2}y_{1}\\
 \mu_{1}x_{2} & \mu_{2}y_{2}
 \end{bmatrix} \begin{bmatrix}
 ({\x}^{T}{\u})^{m-1}\\
 ({\y}^{T}{\u})^{m-1} 
 \end{bmatrix} = \begin{bmatrix}
 0 \\
 0
 \end{bmatrix}.$$
 Since ${\x}$  and ${\y}$ are linearly independent, we get $({\x}^{T}{\u})^{m-1}=0$ and $({\y}^{T}{\u})^{m-1}=0$. This implies ${\x}^{T}{\u} =0$ and ${\y}^{T}{\u}=0$. Since ${\x}$  and ${\y}$ are linearly independent, we must have ${\u} = {\bf 0}$,  a contradiction. Therefore exactly one component of ${\u}$ is non-zero.  
\end{proof}

\begin{rem}\rm 
It is worth pointing out that if ${\A} \in {\mathrm{S}}(m,2)$ with $\mathop{Sym(\A)} =2$ and ${\A} = {\mu_{1}}{\x}^{\bigotimes {m}} + {\mu_{2}}{\y}^{\bigotimes {m}}$ is the symmetric outer product decomposition of ${\A}$, for some non-zero scalars $\mu_{1}, \mu_{2}$ and non-zero vectors ${\x},{\y}$ in ${\mathbb{R}}^{2}$, then depending upon the order of ${\A}$ is even or odd, the decomposition can always be reduced into the form
\begin{enumerate}
\item[\rm(i)] ${\A}= {\x}^{\bigotimes m} + {\y}^{\bigotimes m}$, if $m$ is odd.   
\item[\rm(ii)] ${\A}= {\x}^{\bigotimes m} + {\y}^{\bigotimes m}$ or ${\A}= {\x}^{\bigotimes m} - {\y}^{\bigotimes m}$ or ${\A}= -{\x}^{\bigotimes m} - {\y}^{\bigotimes m}$, when $m$ is even.
\end{enumerate}
For the sake of simplicity, in the subsequent discussions we always consider the decomposition of ${\A}$ into these forms.  
\end{rem}

\begin{lemma}\rm \label{notR}
Let $m$ be odd and ${\A} \in {\mathrm{S}}(m,2)$ with $\mathop{Sym(\A)} =2$. Suppose that ${\A} = {\x}^{\bigotimes m} + {\y}^{\bigotimes m}$ is the symmetric outer product decomposition of ${\A}$, for some ${\x}$ and ${\y}$ in $\mathbb{R}^{2}$, and denote ${\bf A} =\begin{bmatrix}
{\x}  &{\y}
\end{bmatrix} $. If ${\A}$ is not an ${\mathrm{R}_{0}}$-tensor, then ${\bf A}$ can take the following form
\begin{enumerate}
\item[\rm (i)]$\begin{bmatrix}
x_{1} & -x_{1} \\ 
x_{2} & y_{2}
\end{bmatrix}$, where $x_{1} \neq 0$ and $x_{2}+y_{2}>0$, or
\item[\rm (ii)] $\begin{bmatrix}
x_{1} & y_{1} \\ 
x_{2} & -x_{2}
\end{bmatrix}$, where $x_{2} \neq 0$ and $x_{1}+y_{1}>0$.
\end{enumerate} 
\end{lemma}
\begin{proof}
Given that $m$ is an odd number and ${\A} \in {\mathrm{S}}(m,2)$ with $\mathop{Sym(\A)} =2$. Let ${\A} = {\x}^{\bigotimes m} + {\y}^{\bigotimes m}$ be the symmetric outer product decomposition of ${\A}$, where ${\x}$ and ${\y}$ in $\mathbb{R}^{2}$. From Lemma \ref{symrank2}, we have ${\x}$ and ${\y}$ are linearly independent. 
 Suppose that $\A$ is not ${\mathrm {R}_{0}}$, so there exists a non-zero solution of $\T({\A},{\bf 0})$, say ${\u}$. This implies
 \begin{equation}\label{sol}
{\bf u} \geq {\bf 0},~{\A}{\u}^{m-1} \geq {\bf 0},~\text{and}~{\u}^{T}({\A}{\u}^{m-1}) = 0.
\end{equation}
  From Lemma \ref{nonzero}, it follows that ${\u}$  has exactly one non-zero component. We now consider the following cases: 
  
 (C1): Let ${\u}=({u_{1}},0)^{T}; u_{1}>0$. From Eq.(\ref{sol}), we get  $({\A}{\u}^{m-1})_{1} = 0$ and  $({\A}{\u}^{m-1})_{2} \geq 0$. This implies
 \begin{equation}
 \begin{aligned}
 x_{1}({\x}^{T}{\u})^{m-1} + y_{1}({\y}^{T}{\u})^{m-1}=0\implies 
 u_{1}^{m-1}(x_{1}^{m}+y_{1}^{m})=0.
 \end{aligned}
 \end{equation}
 Since $u_{1}>0$, therefore $x_{1}^{m}+y_{1}^{m} = 0$. Due to $m$ being odd, we get $y_{1} = -x_{1}$. This implies $x_{1} \neq 0$, because ${\x}$ and ${\y}$ are linearly independent.  Now $({\A}{\u}^{m-1})_{2} \geq {0}$ implies that $x_{2} ({\x}^{T}{\u})^{m-1} + y_{2}({\y}^{T}{\u})^{m-1} \geq {0}$. Since $u_{2}=0$, $y_{1} = -x_{1}$, and $m$ is odd,  this gives $u_{1}^{m-1} x_{1}^{m-1}(x_{2}+y_{2}) \geq 0$. Due to ${\x}$ and ${\y}$ being linearly independent, we get $x_{2}+y_{2} \neq 0$. Therefore $x_{2}+y_{2} > 0$, because $u_{1} > 0$ and $m$ is odd. Hence $${\bf A} = \begin{bmatrix}
x_{1} & -x_{1} \\ 
x_{2} & y_{2}
\end{bmatrix},~\text{where}~ x_{1} \neq 0~\text{and}~ x_{2}+y_{2}>0.$$

(C2): Let  ${\u}=(0,u_{2})^{T}; u_{2}>0$. Proceeding in the same way as (C1), it can be established that $${\bf A} = \begin{bmatrix}
x_{1} & y_{1} \\ 
x_{2} & -x_{2}
\end{bmatrix},~\text{where}~ x_{2} \neq 0~\text{and}~x_{1}+y_{1}>0.$$
Hence our conclusion follows.
\end{proof}

\begin{theorem} \rm\label{Q}
Let $m$ be odd and ${\A} \in {\mathrm{S}}(m,2)$ with $\mathop{Sym(\A)} =2$. If ${\A}$ is a ${\mathrm {Q}}$-tensor, then ${\A}$ is an ${\mathrm{R}_{0}}$-tensor.
\end{theorem}
\begin{proof}
Given that $m$ is an odd number and ${\A} \in {\mathrm{S}}(m,2)$ with $\mathop{Sym(\A)} =2$. From Theorem \ref{decomposition}, there exists non-zero vectors ${\x},{\y}$ in ${\mathbb{R}}^{2}$ such that ${\A} = {\x}^{\bigotimes m} + {\y}^{\bigotimes m}$. From Lemma \ref{symrank2}, we get ${\x}$ and ${\y}$ are linearly independent. Let us denote ${\bf A} = \begin{bmatrix}
{\x} & {\y}
\end{bmatrix}$. We will prove the contrapositive. Suppose that ${\A}$ is not ${\mathrm{R}_{0}}$. From Lemma \ref{notR}, it follows that 
\begin{enumerate}
\item [\rm (i)] ${\bf A} = \begin{bmatrix}
x_{1} & -x_{1} \\ 
x_{2} & y_{2}
\end{bmatrix},~\text{where}~ x_{1} \neq 0~\text{and}~ x_{2}+y_{2}>0,~\text{or}$
\item[\rm(ii)]${\bf A} = \begin{bmatrix}
x_{1} & y_{1} \\ 
x_{2} & -x_{2}
\end{bmatrix},~\text{where}~x_{2} \neq 0~\text{and}~ x_{1}+y_{1}>0.$
\end{enumerate} 
Suppose that ${\bf A}$ has the form $\begin{bmatrix}
x_{1} & -x_{1} \\ 
x_{2} & y_{2}
\end{bmatrix}$, where $x_{1} \neq 0$ and $x_{2}+y_{2}>0$. Let ${\q} = \begin{bmatrix}
-\lvert x_{1} \rvert \\
\lvert x_{2} \rvert + \lvert y_{2} \rvert
\end{bmatrix}$ be a vector in ${\mathbb R}^{2}$. Let ${\u} = (u_{1},u_{2})^{T} \in \sol({\A},{\q})$, then we have
\begin{equation}\label{solution}
{\u} \geq {\bf 0},~{\A}{\u}^{m-1}+{\q} \geq {\bf 0}~\text{and}~{\u}^{T}({\A}{\u}^{m-1}+{\q}) = 0.
\end{equation}
We claim that $u_{2} > 0$. If not, then 
\begin{equation*}
\begin{aligned}
({\A}{\u}^{m-1}+{\q})_{1} &= x_{1}({\x}^{T}{\u})^{m-1} - x_{1}({\y}^{T}{\u})^{m-1} + q_{1}\\
&= x_{1}(x_{1}u_{1})^{m-1}-x_{1}(-x_{1}u_{1})^{m-1}-\lvert x_{1}\rvert\\
&=-\lvert x_{1}\rvert < 0,~\text{a contradiction}.
\end{aligned}
\end{equation*}
Therefore $u_{2}>0$. Using $u_{2} > 0$, we now prove that $u_{1}>0$. If not, suppose $u_{1} = 0$. Since $u_{2}>0$, from Eq.(\ref{solution}), we get $({\A}{\u}^{m-1}+{\q})_{2} = 0$. Observe that
\begin{equation*}
\begin{aligned}
({\A}{\u}^{m-1}+{\q})_{2} &= x_{2}({\x}^{T}{\u})^{m-1}+y_{2}({\y}^{T}{\u})^{m-1}+q_{2}\\
 &= x_{2}(x_{2} u_{2})^{m-1}+y_{2}(y_{2}u_{2})^{m-1}+ \lvert x_{2} \rvert + \lvert y_{2} \rvert\\
 &=  u_{2}^{m-1}(x_{2}^{m}+y_{2}^{m})+\lvert x_{2} \rvert + \lvert y_{2} \rvert.
\end{aligned}
\end{equation*}
Therefore $({\A}{\u}^{m-1}+{\q})_{2} = 0$ gives us 
\begin{equation}\label{positivity} 
 u_{2}^{m-1}(x_{2}^{m}+y_{2}^{m})+\lvert x_{2} \rvert + \lvert y_{2} \rvert = 0.
\end{equation} 
As $m$ is odd and $x_{2} + y_{2} > 0$, we get $x_{2}^{m}+y_{2}^{m} > 0$. Since each term in Eq.(\ref{positivity}) is positive, we get a contradiction. Hence $u_{1}>0$. Therefore ${\u} > {\bf 0}$. From Eq.(\ref{solution}), we get ${\A}{\u}^{m-1}+{\q} = {\bf 0}$. This can be written as 
$$\begin{bmatrix}
 x_{1} & -x_{1}\\
 x_{2} & y_{2}
 \end{bmatrix} \begin{bmatrix}
 ({\x}^{T}{\u})^{m-1}\\
 ({\y}^{T}{\u})^{m-1} 
 \end{bmatrix} = \begin{bmatrix}
 \lvert x_{1} \rvert \\
 -(\lvert x_{2} \rvert + \lvert y_{2} \rvert)
 \end{bmatrix}.$$
Since ${\x}$ and ${\y}$ are linearly independent, we have 
\begin{equation*}
\begin{bmatrix}
 ({\x}^{T}{\u})^{m-1}\\
 ({\y}^{T}{\u})^{m-1} 
 \end{bmatrix} = \frac{1}{x_{1}(x_{2}+y_{2})} \begin{bmatrix}
 y_{2} & x_{1} \\ 
 -x_{2} & x_{1}
\end{bmatrix}  \begin{bmatrix}
\lvert x_{1} \rvert \\
 -(\lvert x_{2} \rvert + \lvert y_{2} \rvert)
\end{bmatrix}.  
\end{equation*}
This implies  
\begin{equation}\label{z1}
\begin{bmatrix}
 ({\x}^{T}{\u})^{m-1}\\
 ({\y}^{T}{\u})^{m-1} 
 \end{bmatrix} = \frac{1}{x_{1}(x_{2}+y_{2})}\begin{bmatrix}
 y_{2}\lvert x_{1} \rvert -x_{1}(\lvert x_{2} \rvert + \lvert y_{2} \rvert) \\
 -x_{2}\lvert x_{1} \rvert-x_{1}(\lvert x_{2} \rvert + \lvert y_{2} \rvert)
 \end{bmatrix}.
\end{equation}
We now consider the following cases:\\
(C1): When $x_{1}>0$. From  Eq.(\ref{z1}), we have
\begin{equation}\label{z2}
\begin{bmatrix}
 ({\x}^{T}{\u})^{m-1}\\
 ({\y}^{T}{\u})^{m-1} 
 \end{bmatrix} = \frac{x_{1}}{x_{1}(x_{2}+y_{2})}  \begin{bmatrix}
 y_{2} - (\lvert x_{2} \rvert + \lvert y_{2} \rvert)\\
 -(x_{2}+\lvert x_{2} \rvert + \lvert y_{2} \rvert)
 \end{bmatrix}. 
\end{equation}
Since $x_{2}+y_{2}>0$, from Eq.(\ref{z2}), we get $({\y}^{T}{\u})^{m-1}<0$, which is a contradiction because $m$ is odd. \\
(C2): When $x_{1}<0$. From  Eq.(\ref{z1}), we have
\begin{equation}\label{z3}
\begin{bmatrix}
 ({\x}^{T}{\u})^{m-1}\\
 ({\y}^{T}{\u})^{m-1} 
 \end{bmatrix} = \frac{x_{1}}{x_{1}(x_{2}+y_{2})}  \begin{bmatrix}
-( y_{2} + (\lvert x_{2} \rvert + \lvert y_{2} \rvert)\\
 x_{2}-(\lvert x_{2} \rvert + \lvert y_{2} \rvert)
 \end{bmatrix}. 
\end{equation}
Since $x_{2}+y_{2}>0$, from Eq.(\ref{z3}), we get $({\x}^{T}{\u})^{m-1}<0$, which is a contradiction because $m$ is odd. From both of the above cases, we get  ${\A}$ is not a ${\rm Q}$-tensor. 

Now assume that $${\bf A} = \begin{bmatrix}
x_{1} & y_{1} \\ 
x_{2} & -x_{2}
\end{bmatrix},~\text{where}~x_{2} \neq 0~\text{and}~ x_{1}+y_{1}>0.$$
Let ${\bf P} = \begin{bmatrix}
0&1\\
1&0
\end{bmatrix}$. Then $$\hat{\bf A} =  {\bf P}{\bf A} = \begin{bmatrix}
{x_{2}}  &  -{x_{2}} \\
{x_{1}}  &  {y_{1}}
\end{bmatrix},~\text{where}~x_{2} \neq 0~\text{and}~ x_{1}+y_{1}>0.$$
From the above proof, it follows that ${\hat{\A}}$ is not ${\rm Q}$, where ${\hat{\A}} = {({\bf P}{\x})}^{\bigotimes m} + {({\bf P}{\y})}^{\bigotimes m}$.  From Corollary \ref{not}, we get ${\A}$ is not ${\rm Q}$.

In summary, we have proved that ${\A}$ is not a ${\rm Q}$-tensor. Hence proved.  
\end{proof}

We now prove that for an even ordered tensor ${\A}$ in ${\mathrm S}(m,2)$ with $\mathop {Sym (\A)} =2$, the implication ${\A}$ is a $\mathrm{Q}$-tensor then $\A$ is an ${\mathrm R}_{0}$-tensor, is true. Note that if ${\A} \in \mathrm{S}(m,2)$ is an even ordered tensor with $\mathop {Sym (\A)} =2$, then there exist non-zero vectors ${\x}$ and ${\y}$ in ${\mathbb R}^{2}$ such that either ${\A} = {\x}^{\bigotimes m} + {\y}^{\bigotimes m}$ or ${\A} = {\x}^{\bigotimes m} - {\y}^{\bigotimes m}$ or ${\A} = -[{\x}^{\bigotimes m} + {\y}^{\bigotimes m}]$. Also these vectors ${\x}$ and ${\y}$ are linearly independent. For the first case, when ${\A} = {\x}^{\bigotimes m} + {\y}^{\bigotimes m}$, it is immediate that $\A$ is a positive definite tensor (i.e. ${\u}^{T}{\A}{\u}^{m-1} > {0},\forall~{\u} \in {\mathbb R}^{n}$) and hence from \citep[][Theorem 5.3]{Palpandi}, $\T({\A},{\q})$ has a unique solution for each ${\q}$ in ${\R}$. Also note that ${\A}$ can not be a ${\rm Q}$-tensor when ${\A} = -[{\x}^{\bigotimes m} + {\y}^{\bigotimes m}]$, because for a vector ${\q} \in {\R}$ with $q_{i}<{0}$, for each $i \in [n]$, ${\T}({\A},{\q})$ does not have a solution. Therefore in the subsequent discussion, we prove our result only when $\A$ takes the form ${\A} = {\x}^{\bigotimes m} - {\y}^{\bigotimes m}$. To proceed further, we need the following lemma.
\begin{lemma}\rm \label{evenm}
Let $m$ be even and ${\A} \in {\mathrm{S}}(m,2)$ with $\mathop{Sym(\A)} =2$. Suppose that ${\A} = {\x}^{\bigotimes m} - {\y}^{\bigotimes m}$ is the symmetric outer product decomposition of ${\A}$, for some ${\x}$ and ${\y}$ in $\mathbb{R}^{2}$, and denote ${\bf A} =\begin{bmatrix}
{\x}  &{\y}
\end{bmatrix} $. If ${\A}$ is not an ${\mathrm{R}_{0}}$-tensor, then ${\bf A}$ can take any one of the following form, 
\begin{enumerate}
\item[\rm(i)] $\begin{bmatrix}
x_{1} &  \alpha x_{1} \\ 
x_{2} & y_{2}
\end{bmatrix}$ with  $x_{1} (x_{2} -\alpha y_{2}) > 0$, or

\item[\rm(ii)] $\begin{bmatrix}
x_{1} & y_{1} \\ 
x_{2} & \alpha x_{2}
\end{bmatrix}$ with $ x_{2}(x_{1} - \alpha y_{1})>0$, where ${\alpha} = \pm 1$.
\end{enumerate}
\end{lemma}
\begin{proof}
Let $m$ be an even number and ${\A} \in \mathrm{S}(m,2)$ with $\mathop{Sym(\A)} = 2$. Let ${\A} = {\x}^{\bigotimes m} - {\y}^{\bigotimes m}$ be the symmetric outer product decomposition of ${\A}$, where ${\x}$ and ${\y}$ in $\mathbb{R}^{2}$. From Lemma \ref{symrank2}, we have ${\x}$ and ${\y}$ are linearly independent. 
Let ${\bf A} =\begin{bmatrix}
{\x}  &{\y}
\end{bmatrix}$. 
 Suppose that $\A$ is not ${\mathrm {R}_{0}}$, so there exists a non-zero solution of $\T({\A},{\bf 0})$, say ${\u}$. From Lemma \ref{nonzero}, it follows that ${\u}$  has exactly one non-zero component. We now consider the following cases:
 
 (C1): Let ${\u}=({u_{1}},0)^{T}; u_{1}>0$. This implies $({\A}{\u}^{m-1})_{1} = 0$ and  $({\A}{\u}^{m-1})_{2} \geq 0$. Therefore we get 
 \begin{equation}
 \begin{aligned}
 x_{1}({\x}^{T}{\u})^{m-1} - y_{1}({\y}^{T}{\u})^{m-1}=0\implies 
 u_{1}^{m-1}(x_{1}^{m}-y_{1}^{m})=0,
 \end{aligned}
 \end{equation}
 Since $u_{1}>0$, therefore $x_{1}^{m}-y_{1}^{m} = 0$. Due to $m$ being even, we get $y_{1} = x_{1}$ or $y_{1} = -x_{1}$. We now consider the following two sub cases:
 \begin{enumerate}
 \item[(a):] When  $y_{1} = x_{1}$. Since $({\A}{\u}^{m-1})_{2} \geq 0$, therefore we get $(u_{1} x_{1})^{m-1} (x_{2} -y_{2}) \geq 0$. Due to $u_{1}>0$, ${\x}$ and ${\y}$ being linearly independent, and $m$ being even, we get  $x_{1}(x_{2}-y_{2}) > 0$.
 \item[(b):] When  $y_{1} = -x_{1}$. Since $({\A}{\u}^{m-1})_{2} \geq 0$, therefore we get $(u_{1} x_{1})^{m-1} (x_{2} +y_{2}) \geq 0$. Due to $u_{1}>0$, ${\x}$ and ${\y}$ being linearly independent, and $m$ being even, $x_{1}(x_{2}+y_{2}) > 0$.
 \end{enumerate} 
Therefore ${\bf A}$ has the form $\begin{bmatrix}
x_{1} &  \alpha x_{1} \\ 
x_{2} & y_{2}
\end{bmatrix}$ with  $x_{1} (x_{2} -\alpha y_{2}) > 0$, where $\alpha = \pm 1$.\\

 (C2): Let  ${\u}=(0,u_{2})^{T}; u_{2}>0$. Proceeding in the same way as (C1), we can obtain the required result.
\end{proof}

We now prove our main result.
\begin{theorem}\rm \label{main}
Let $m$ be even and ${\A} \in {\mathrm S}(m,2)$ with $\mathop {Sym(\A)} =2$. If ${\A}$ is a ${\mathrm Q}$-tensor, then ${\A}$ is an ${\mathrm R}_{0}$-tensor.
\end{theorem}
\begin{proof}
Given that $m$ is an even number and ${\A} \in {\mathrm S}(m,2)$ with $\mathop {Sym (\A)} =2$. We will prove the contrapositive. Assume that ${\A}$ is not ${\mathrm R}_{0}$. So there exists two linearly independent vectors ${\x}$ and ${\y}$ in ${\mathbb R}^{2}$ such that 
$${\A} = {\x}^{\bigotimes m} - {\y}^{\bigotimes m}.$$
Let ${\bf A} =\begin{bmatrix}
{x_{1}}  &  {y_{1}} \\
{x_{2}}  &  {y_{2}}
\end{bmatrix}$.  From Lemma \ref{evenm}, it follows that 
\begin{enumerate}
\item[\rm(i)] $ {\bf A} = \begin{bmatrix}
x_{1} &  \alpha x_{1} \\ 
x_{2} & y_{2}
\end{bmatrix}~\text{with}~x_{1} (x_{2} -\alpha y_{2}) > 0,~\text{or}$ 
\item[\rm(ii)] $ {\bf A} = \begin{bmatrix}
x_{1} & y_{1} \\ 
x_{2} & \alpha x_{2}
\end{bmatrix}~\text{with}~x_{2}(x_{1} - \alpha y_{1})>0,~\text{where} ~\alpha = \pm 1.$
\end{enumerate}
Suppose that ${\bf A} =\begin{bmatrix}
x_{1} &  \alpha x_{1} \\ 
x_{2} & y_{2}
\end{bmatrix}$ and  $x_{1} (x_{2} -\alpha y_{2}) > 0$, where $\alpha = \pm 1$.
 We proceed by considering the following cases for $x_{2}$  and $y_{2}$.\\
(C1): When $\lvert x_{2} \rvert \geq \lvert y_{2} \rvert$. Then the entries of the tensor ${\A}$ are given as
\begin{equation}\label{nonnegative}
\begin{aligned}
a_{i_{1}i_{2}...i_{m}}&=x_{i_{1}} x_{i_{2}}...x_{i_{m}}-y_{i_{1}}y_{i_{2}}...y_{i_{m}}\\
&=  x_{1}^{k} x_{2}^{m-k} - (\alpha x_{1})^{k}y_{2}^{m-k}\\
&= x_{1}^{k} (x_{2}^{m-k} - (\alpha)^{k} y_{2}^{m-k}),~\text{for some}~k. 
\end{aligned}
\end{equation}
As we have  $x_{1}(x_{2} -\alpha y_{2})>0$. This implies $x_{1}^{k} (x_{2}^{m-k} - (\alpha )^{k} y_{2}^{m-k}) = x_{1}^{k} (x_{2}^{m-k} -\alpha y_{2}^{m-k}) > 0$, for all odd $k$. Also, when $k$ is even, we get $x_{1}^{k} (x_{2}^{m-k} - (\alpha )^{k} y_{2}^{m-k})=x_{1}^{k} (x_{2}^{m-k} - y_{2}^{m-k}) \geq 0$ because $\lvert x_{2} \rvert \geq \lvert y_{2} \rvert$. Therefore from Eq.(\ref{nonnegative}), we have $a_{i_{1}i_{2}...i_{m}} \geq 0$, for each $i_{1},i_{2},...,i_{m}$. So ${\A}$ is a nonnegative tensor with $a_{11...1} = x_{1}^{m} - (\alpha x_{1})^{m} =0$. By Theorem 3.5 in \cite{Yisheng}, it implies that ${\A}$ is not a ${\mathrm Q}$-tensor.\\
(C2): When $\lvert x_{2} \rvert < \lvert y_{2} \rvert$. Let ${\q} = \begin{bmatrix}
\lvert x_{1} \rvert \\
-t
\end{bmatrix}$, where $t>0$, be a vector in ${\mathbb R}^{2}$. Let ${\u} = (u_{1},u_{2})^{T} \in \sol({\A},{\q})$, then we have
\begin{equation}\label{solution2}
{\u} \geq {\bf 0},~{\A}{\u}^{m-1}+{\q} \geq {\bf 0}~\text{and}~{\u}^{T}({\A}{\u}^{m-1}+{\q}) = 0.
\end{equation}
Note that ${\u}$ can not be a zero vector as ${\q}$ has a negative component. Therefore we have the following possibilities for the vector ${\u} = (u_{1},u_{2})^{T}$.\\
(a): $u_{1} = 0$ and $u_{2} > 0$. From Eq.(\ref{solution2}), we get $({\A}{\u}^{m-1}+{\q})_{2} = 0$. This implies 
\begin{equation}\label{negative}
u_{2}^{m-1}(x_{2}^{m}-y_{2}^{m})-t = 0.
\end{equation}
As $\lvert x_{2} \rvert < \lvert y_{2} \rvert$ and $m$ is even, we get $(x_{2}^{m}-y_{2}^{m})<0$. So each term in Eq.(\ref{negative}) is negative, which is a contradiction.\\
(b): $u_{1}>0$, and $u_{2}=0$. From Eq.(\ref{solution2}), we get $({\A}{\u}^{m-1}+{\q})_{1} = 0$. This implies 
\begin{equation}\label{second}
u_{1}^{m-1}(x_{1}^{m}- (\alpha x_{1})^{m})+{\lvert x_{1} \rvert} = 0.
\end{equation}
As $m$ is even and $x_{1}$ is non-zero, we get a contradiction from Eq.(\ref{second}).\\
(c): $u_{1} >0$ and $u_{2}>0$. From Eq.(\ref{solution2}), we get ${\A}{\u}^{m-1}+{\q} = {\bf 0}$. This implies 
$$\begin{bmatrix}
 x_{1} & -\alpha x_{1}\\
 x_{2} & -y_{2}
 \end{bmatrix} \begin{bmatrix}
 ({\x}^{T}{\u})^{m-1}\\
 ({\y}^{T}{\u})^{m-1} 
 \end{bmatrix} = \begin{bmatrix}
 -\lvert x_{1} \rvert \\
 t
 \end{bmatrix}.$$
Since ${\x}$ and ${\y}$ are linearly independent, we have
$$\begin{bmatrix}
 ({\x}^{T}{\u})^{m-1}\\
 ({\y}^{T}{\u})^{m-1} 
 \end{bmatrix} = \frac{1}{x_{1}(-y_{2}+ \alpha x_{2})}\begin{bmatrix}
 -y_{2} & \alpha x_{1} \\
 -x_{2} & x_{1}
\end{bmatrix} \begin{bmatrix}
-\lvert x_{1} \rvert \\
t
\end{bmatrix} = \frac{1}{x_{1}(-y_{2} + \alpha x_{2})} \begin{bmatrix}
 y_{2}\lvert x_{1} \rvert + \alpha x_{1}t \\
 x_{2}\lvert x_{1} \rvert +x_{1}t
 \end{bmatrix}.$$
 As $m$ is even, this can be equivalently written as
 $$\begin{bmatrix}
 x_{1} &  x_{2} \\
 \alpha x_{1} & y_{2}
 \end{bmatrix} \begin{bmatrix}
 u_{1} \\ u_{2}
 \end{bmatrix} = \frac{1}{\{x_{1}(-y_{2}+ \alpha x_{2})\}^{\frac{1}{m-1}}} \begin{bmatrix}
 (y_{2}\lvert x_{1} \rvert + \alpha x_{1}t)^\frac{1}{m-1} \\
 (x_{2}\lvert x_{1} \rvert +x_{1}t)^\frac{1}{m-1} 
 \end{bmatrix}.$$
 Since ${\x}$ and ${\y}$ are linearly independent, upon simplifying further, we get
 \begin{equation*}
 \begin{bmatrix}
 u_{1}\\
 u_{2}
 \end{bmatrix} =  \frac{(-1)}{\{x_{1}(-y_{2} +\alpha x_{2})\}^{\frac{m}{m-1}}} \begin{bmatrix}
 y_{2} & -x_{2} \\
 -\alpha x_{1} & x_{1} 
 \end{bmatrix} \begin{bmatrix}
 (y_{2}\lvert x_{1} \rvert + \alpha x_{1}t)^\frac{1}{m-1} \\
 (x_{2}\lvert x_{1} \rvert +x_{1}t)^\frac{1}{m-1} 
 \end{bmatrix}
 \end{equation*}
 This implies
\begin{equation}\label{important}
\begin{bmatrix}
u_{1}\\
u_{2}
\end{bmatrix} = \frac{(-1)}{{\{x_{1}(x_{2}- \alpha y_{2})\}}^{\frac{m}{m-1}}} \begin{bmatrix}
 y_{2}(y_{2}\lvert x_{1} \rvert + \alpha x_{1}t)^{\frac{1}{m-1}} - x_{2}(x_{2}\lvert x_{1} \rvert +x_{1}t)^{\frac{1}{m-1}}\\
 -\alpha x_{1}(y_{2}\lvert x_{1} \rvert + \alpha x_{1}t)^{\frac{1}{m-1}} + x_{1}(x_{2}\lvert x_{1} \rvert +x_{1}t)^{\frac{1}{m-1}}
\end{bmatrix}. 
\end{equation}
Since we have $x_{1}(x_{2} -\alpha y_{2})>0$, therefore the denominator in Eq.(\ref{important}) is always positive. The following possibilities can be taken into consideration:\\
(i): If $x_{1}>0$, then  $x_{2}-\alpha y_{2}>0$. This implies $(x_{2}+t)> \alpha(y_{2}+ \alpha t)$. Also, from Eq.(\ref{important}), we have
\begin{equation}\label{lasteq}
 u_{2} = \frac{(-1)(x_{1}^\frac{m}{m-1})}{{\{x_{1}(x_{2} - \alpha y_{2})\}}^{\frac{m}{m-1}}} \{(x_{2}+t)^\frac{1}{m-1} - \alpha (y_{2} + \alpha t)^\frac{1}{m-1}\}.
 \end{equation}
(ii): If $x_{1}<0$, then  $x_{2}-\alpha y_{2}<0$. This implies $(t-x_{2}) > \alpha (\alpha t - y_{2})$. Also, from Eq.(\ref{important}), we have
\begin{equation}\label{secondlast}
u_{2} = \frac{(-1)(x_{1}^\frac{m}{m-1})}{{\{x_{1}(x_{2}-\alpha y_{2})\}}^{\frac{m}{m-1}}} \{(t-x_{2})^\frac{1}{m-1} - \alpha (\alpha t-y_{2})^\frac{1}{m-1}\}.
\end{equation}
As $m~(>1)$ is even, the function $f(a) = a^{\frac{1}{m-1}
}$ is a strictly increasing function, for all real numbers $a$. Hence from Eqs. (\ref{lasteq}) and (\ref{secondlast}), we get $u_{2}< 0$, a contradiction.
 Therefore we get contradiction in each of the above cases implying that $\A$ is not a ${\mathrm Q}$-tensor.  Now assume that $${\bf A} = \begin{bmatrix}
x_{1} & y_{1} \\ 
x_{2} & \alpha x_{2}
\end{bmatrix},~x_{2}(x_{1} - \alpha y_{1})>0,~\text{where}~ \alpha = \pm 1.$$ 
Then by the same procedure as in Theorem \ref{Q}, we can show that ${\A}$ is not a ${\rm Q}$-tensor. Hence our conclusion follows. 
\end{proof}

After combining Theorems \ref{Q} and \ref{main}, we have 
\begin{theorem}\rm \label{last}
Let ${\A} \in \mathrm{S}(m,2)$ and $\mathop{Sym}({\A})=2$. If ${\A}$ is a ${\rm Q}$-tensor, then ${\A}$ is an ${\rm R}_{0}$-tensor.
\end{theorem}
The converse of the above theorem may not be true even in the case of $m =2$. For example, let ${\bf A} = - {\bf I} \in {\mathbb R^{2 \times 2}}$. So ${\bf A}$ is a symmetric matrix with its rank equal to $2$. It can be seen easily that ${\bf A}$ is an ${\rm R}_{0}$-matrix but ${\bf A}$ is not ${\rm Q}$ as the $\L({\bf A},{\q})$ does not have a solution, for ${\q}  =(-1,-1)^{T}$.\\

The following corollary is a by-product of Theorem \ref{last},
and a result in \cite[Proposition 5.5]{comon} which states that for ${\A} \in {\rm S}(m,n)$, if $\mathop {Sym}({\A}) = 1$ or $2$, then $rank({\A})= \mathop {Sym}({\A})$. 

\begin{corollary}\rm
Let ${\A} \in {\rm S}(m,2)$ and $rank({\A}) =2$. If ${\A}$ is a ${\rm Q}$-tensor, then ${\A}$ is an ${\rm R}_{0}$-tensor. 
\end{corollary}

\section{Conclusion}
In this paper, we establish the equivalence of ${\rm Q}$-tensor, ${\rm S}$-tensor, positive tensor and ${\rm R}_{0}$-tensor, for the class of rank one symmetric tensors. Furthermore, we discuss the relationship between ${\rm Q}$-tensor and ${\rm R}_{0}$-tensor, for the class of symmetric tensors having symmetric rank  equal to $2$.

It is evident from the literature \cite{{Blekherman},{comon}} that the symmetric rank of a tensor can exceed its dimension, however sharp upper bounds are provided for the case of $n =2$. The question whether the ${\rm Q}$ implies ${\rm R}_{0}$ holds, when the symmetric rank is greater than equal to $3$, remain yet to be answered. This is an issue to be further studied.

\end{document}